\documentclass[11pt]{amsart}

\usepackage{fullpage}
\usepackage{graphicx}
\usepackage{amsfonts}
\usepackage{amssymb}
\usepackage{amsmath}
\usepackage{color}
\def\jump#1{{[\![#1]\!]}}
\def\bn{{\bf n}}
\def\bRh{\mathbb{Q}_h}
\def\T{\mathcal{T}}
\newtheorem{lemma}{Lemma}[section]

\begin{document}
\title{Weak Galerkin Finite Element methods for Parabolic Equations}
\author{Qiaoluan H. Li}
\address{Department of Mathematics, Towson University, Towson, MD 22152} \email{qli@towson.edu}
\author{Junping Wang}
\address{Division of Mathematical Sciences, National Science
Foundation, Arlington, VA 22230} \email{jwang@nsf.gov}
\thanks{The research of Wang was supported by the NSF IR/D program,
while working at the National Science Foundation. However, any
opinion, finding, and conclusions or recommendations expressed in
this material are those of the author and do not necessarily reflect
the views of the National Science Foundation.}

\begin{abstract}
A newly developed weak Galerkin method is proposed to solve
parabolic equations. This method allows the usage of totally
discontinuous functions in approximation space and preserves the
energy conservation law. Both continuous and discontinuous time weak
Galerkin finite element schemes are developed and analyzed. Optimal
order error estimates in both $H^1$ and $L^2$ norms are established.
Numerical tests are performed and reported.
\end{abstract}



\maketitle

\section{Introduction}
In this paper, we consider the initial-boundary value problem
\begin{eqnarray}\label{continuousP}
               u_t-\nabla\cdot(a\nabla u)  &=& f \quad \mbox{ in } \Omega,\quad t\in J,\label{heat}\nonumber\\
               u &=& g\quad \mbox{ on }\partial\Omega, \quad  t\in J,\\
               u(\cdot, 0)&=&\psi\quad \mbox{ in } \Omega,\nonumber
             \end{eqnarray}
where $\Omega$ is a polygonal or polyhedral domain in $\mathbb{R}^d$
$(d=\,2,\,3)$ with Lipschitz-continuous boundary $\partial \Omega$,
$J=(0, \, \bar{t}\,]$, and $a\,=\,(a_{ij}(x,t))_{d\times d}\in
[L^\infty(\Omega\times \bar{J})]^{d^2}$ is a symmetric matrix-valued
function satisfying the following property: there exists a constant
$\alpha>0$ such that
$$
\alpha \xi^T\xi\leq \xi^T a\xi,\quad \forall \xi\in \mathbb{R}^d.
$$
For simplicity, we shall concentrate on two-dimensional problems
only (i.e., $d=2$).

For any nonnegative integer $m$, let $H^m(\Omega)$
be the standard Sobolov space, which is the collection of all real-valued
functions defined on $\Omega$ with square integrable derivatives up
to order $m$ with
semi-norm
$$|\psi|_{s,\Omega}\equiv\{\sum_{|\alpha|=s}\int_\Omega\,|\partial^\alpha
\psi|^2\,dx\}^{\frac12}, $$
and norm
$$
\|\psi\|_{m,\Omega}\equiv(\sum_{j=0}^m |\psi|^2_{j,\Omega})^{\frac12}.
$$
For simplicity, we use $\|\cdot\|$ for the $L^2$ norm.

The parabolic problem can be written in a variational form as
follows
\begin{eqnarray}\label{weakform}
(u_t,v)+(a\nabla u,\nabla v)&=&(f,v)\quad \forall v\in H_0^1(\Omega),\,t\in J,\\
u(\cdot, 0)&=&\psi,\nonumber
\end{eqnarray}
where $u$ is called a weak solution if $u\in L^2(0,t;H^1(\Omega))$
and $u_t\in L^2(0,t;H^{-1}(\Omega))$, and if $u=g$ on
$\partial\Omega$.

Parabolic problems have been treated by various numerical methods.
For finite element methods, we refer to \cite{VThomee} and
\cite{AFEM}. Discontinuous Galerkin finite element methods were
studied in \cite{DGM} and \cite{DGM2}. In \cite{FVM} and
\cite{FVM2}, finite volume methods were presented, which were based
on the integral conservation law rather than the differential
equation. The integral conservation law was then enforced for small
control volumes defined by the computational mesh.

The goal of this paper is to consider a newly developed weak
Galerkin (WG) finite element method for parabolic equation which is
based on the definition of a discrete weak gradient operator
proposed in \cite{JW_WG}. In this numerical method, the analysis can
be done by using the framework of Galerkin methods, and totally
discontinuous functions are allowed to be used as the approximation
space. Furthermore, the approximation results also satisfy the
energy conservation law.

The rest of this paper is organized as follows. In Section 2, we
introduce some notation and establish a continuous time and
discontinuous time weak Galerkin finite element scheme for the
parabolic initial boundary-value problem (\ref{continuousP}). In
Section 3, we prove the energy conservation law of the weak Galerkin
approximation. Optimal order error estimates in both $H^1$ and $L^2$
norms are proved in Section 4. The paper is concluded with some
numerical experiments to illustrate the theoretical analysis in
Section 5.

\section{The Weak Galerkin Method}
In this section we design a continuous time and a discontinuous time
weak Galerkin finite element scheme for the initial-boundary value
problem (\ref{continuousP}). We consider the space of discrete weak
functions and the discrete weak operator introduced in \cite{JW_WG}.
Let $\mathcal{T}_h$ be a quasiuniform family of triangulations of a
plane domain $\Omega$ and $T$ be each triangle element with
$\partial T$ as its boundary. For each $T \in \mathcal{T}_h$, let
$P_j(T)$ be the set of polynomials on $T$ with degree no more than
$j$, and $P_l(\partial T)$ be the set of polynomials on $\partial T$
with degree no more than $l$, respectively. Let $\hat{P}_j(T)$ be
the set of homogeneous polynomials on $T$ with degree $j$. The weak
finite element space $S_h(j,l)$ is defined by
$$
S_h(j,l):=\{v=\{v_0,v_b\}\,:\,\,v_0\in P_j(T),\,v_b\in
P_l(e)\,\mbox{ for all edge } e\subset \partial T,\ T\in
\mathcal{T}_h\}.
$$
Denote by $S^0_h(j,l)$ the subspace of $S_h(j,l)$ with vanishing
boundary value on $\partial \Omega$; i.e.,
$$
S^0_h(j,l):=\{v=\{v_0,v_b\}\in S_h(j,l),\,v_b|_{\partial T\cap\partial \Omega}=0\,\mbox{ for all } T\in \mathcal{T}_h\}.
$$

Let $\sum_h=\{ {\bf q}\in [L^2(\Omega)]^2\,:\,{\bf q}|_T\in V(T,r)
\mbox{ for all } T\in \mathcal{T}_h\}$, where $V(T,r)$ is a subspace
of the set of vector-valued polynomials of degree no more than $r$
on $T$. For each $v=\{v_0,v_b\}\in S_h(j,l)$, the discrete weak
gradient $\nabla_d v\in \sum_h$ of $v$ on each element $T$ is given
by the following equation:
\begin{equation}\label{WGoperator}
\int_T \nabla_{d}v\cdot {\bf q}\,dT=-\int_T v_0\nabla\cdot {\bf
q}\,dT+\int_{\partial T} v_b{\bf q}\cdot \bn\,ds,\quad \forall {\bf
q}\in V(T,r),
\end{equation}
where $\bn$ is the outward normal direction of $\partial T$. It is
easy to see that this discrete weak gradient is well-defined.

To investigate the approximation properties of the discrete weak
spaces $S_h(j,l)$ and $\sum_h$, we use three projections in this
paper. The first two are local projections defined on each triangle
$T$: one is $Q_h u=\{Q_0 u,Q_b u\}$, the $L^2$ projection of
$H^1(T)$ onto $P_j(T)\times P_l(\partial T)$ and another is $\bRh$,
the $L^2$ projection of $[L^2(T)]^2$ onto $V(T,r)$. The third
projection $\Pi_h$ is assumed to exist and satisfy the following
property: For ${\bf q}\in H(div,\Omega)$ with mildly added
regularity, $\Pi_h {\bf q}\in H(div,\Omega)$ such that $\Pi_h {\bf
q}\in V(T,r)$ on each $T\in \mathcal{T}_h$, and
$$
(\nabla \cdot {\bf q},v_0)_T=(\nabla \cdot \Pi_h {\bf
q},v_0)_T,\quad \forall v_0\in P_j(T).
$$

It is easy to see the following two useful identities:
\begin{equation}\label{identity1}
\nabla_d(Q_h w)=\bRh(\nabla w),\quad \forall w\in H^1(T),
\end{equation}
and for any ${\bf q}\in H(div,\Omega)$
\begin{equation}\label{identity2}
\sum_{T\in\mathcal{T}_h}(-\nabla\cdot{\bf q},v_0)_T=
\sum_{T\in\mathcal{T}_h}(\Pi_h{\bf q},\nabla_d v)_T, \quad\forall
v=\{v_0,v_b\}\in S_h^0(j,l).
\end{equation}

The discrete weak spaces $S_h(j,l)$ and $\sum_h$ need to possess
some good approximation properties in order to provide an acceptable
finite element scheme. In \cite{JW_WG}, the following two criteria
were given as a general guideline for their construction:

\begin{description}
\item[\bf P1] For any $v\in S_h(j,l)$, if $\nabla_d
v=0$ on $T$, then one must have $v\equiv {\it constant}$ on $T$;
i.e., $v_0=v_b={\it constant}$ on $T$.
\item[\bf P2] For any $u\in H^1_0(\Omega)\cap
H^{m+1}(\Omega)$, where $0\leq m\leq j+1$, the discrete weak
gradient of the $L^2$ projection $Q_h u$ of $u$ in the discrete weak
space $S_h(j,l)$ provides a good approximation of $\nabla u$; i.e.,
$\|\nabla_d (Q_h u)-\nabla u\|\leq Ch^m\|u\|_{m+1}$ holds true.
\end{description}

\medskip
Examples of $S_h(j,l)$ and $\sum_h$ satisfying the conditions {\bf
P1} and {\bf P2} can be found in \cite{JW_WG}. For example, with
$V(T,r=j+1)=[P_{j}(T)]^2+\hat{P}_j(T){\bf x}$ being the usual
Raviart-Thomas element \cite{RT} of order $j$, one may take
equal-order elements of order $j$ for $S_h(j,l)$ in the interior and
the boundary of each element $T$. The key of using the
Raviart-Thomas element for $V(T,r)$ is to ensure the condition {\bf
P1}. The condition {\bf P2} was satisfied by the commutative
property (\ref{identity1}) which has been established in
\cite{JW_WG}. It was shown later in \cite{wy-mixed, JW2} that the
condition {\bf P1} can be circumvented by a suitable stabilization
term. Consequently, the selection of $V(T,r)$ and $S_h(j,l)$ becomes
very flexible and robust in practical computation. It even allows
the use of finite elements of arbitrary shape.

The main idea of the weak Galerkin method is to use the discrete
weak space $S_h(j,l)$ as testing and trial spaces and replace the
classical gradient operator by the weak gradient operator $\nabla_d$
in (\ref{weakform}).

First, we pose the continuous time weak Galerkin finite element
method, based on the weak Galerkin operator (\ref{WGoperator}) and
weak formulation (\ref{weakform}), which is to find
$u_h(t)=\{u_0(\cdot, t),u_b(\cdot, t)\}$, belonging to $S_h(j,l)$
for $t\geq 0$, satisfying $u_b=Q_b g \mbox{ on } \partial \Omega,
t>0$ and $u_h(0)=Q_h \psi \mbox{ in } \Omega$, and the following
equation
\begin{equation}\label{semiWG}
((u_{h})_t,v_0)+a(u_h,v)=(f,v_0)\quad \forall v=\{v_0,v_b\}\in
S_h^0(j,l),\,t>0,
\end{equation}
where $a(\cdot,\cdot)$ is the weak bilinear form defined by
$$
a(w,v)=(a\nabla_d w,\nabla_d v),
$$
which is assumed to be bounded and coercive, i.e., for constant $\alpha,\beta,\gamma>0$
$$
|a(u,v)|\leq \beta \|\nabla_d u\|\|\nabla_d v\|,
$$
$$
a(u,u)\geq \alpha \|\nabla_d u\|^2,
$$
and that
$$
|(a_t \nabla_d u,\nabla_d v)|\leq \gamma \|\nabla_d u\|\|\nabla_d v\|.
$$

We now turn our attention to some discrete time Weak Galerkin
procedures. We introduce a time step $k$ and the time levels
$t=t_n=nk$, where $n$ is a nonnegative integer, and denote by
$U^n=U^n_h\in S_h(j,l)$ the approximation of $u(t_n)$ to be
determined. The backward Euler Weak Galerkin method is defined by
replacing the time derivative in (\ref{semiWG}) by a backward
difference quotient, or, if $\bar{\partial} U^n=(U^n-U^{n-1})/k$,
\begin{equation}\label{discreteWG}
(\bar{\partial} U^n,v_0)+a(U^n,v)=(f(t_n),v_0)\quad \forall
v=\{v_0,v_b\}\in S_h^0(j,l),\,n\geq 1,\,U^0=Q_h\psi,
\end{equation}
i.e.,
$$
(U^n,v_0)+ka(U^n,v)=(U^{n-1}+kf(t_n),v_0), \quad \forall
v=\{v_0,v_b\}\in S_h^0(j,l).
$$

\section{Energy Conservation of Weak Galerkin}

This section investigates the energy conservation property of the
weak Galerkin finite element approximation $u_h$. The increase in
internal energy in a small spatial region of the material $K$, i.e.,
control volume, over the time period $[t-\Delta t,t+\Delta t]$ is
given by
$$
\int_K u(x,t+\Delta t)\,dx-\int_K u(x,t-\Delta t)\,dx=\int_{t-\Delta t}^{t+\Delta t}\int_K u_t\,dx\,dt.
$$
  Suppose that a body obeys the heat equation and, in addition, generates its own heat per unit volume at a rate given by a known function $f$ varying in space and time, the change in internal energy in $K$ is accounted for by the flux of heat across the boundaries together with the source energy. By Fourier's law, this is
$$
-\int_{t-\Delta t}^{t+\Delta t}\int_{\partial K} q\cdot \bn \,ds
\,dt +\int_{t-\Delta t}^{t+\Delta t}\int_K f\,dx\,dt,
$$
where $q=-a\nabla u$ is the flow rate of heat energy.
The solution $u$ of the problem (\ref{continuousP}) yields the following integral form of energy conservation:
\begin{equation}\label{energy}
\int_{t-\Delta t}^{t+\Delta t}\int_K u_t\,dx\,dt+\int_{t-\Delta
t}^{t+\Delta t}\int_{\partial K} q\cdot \bn\,ds\,dt=\int_{t-\Delta
t}^{t+\Delta t}\int_K f\,dx\,dt
\end{equation}
where the Green's formula was used. We claim that the numerical approximation from the weak Galerkin finitel element method for (\ref{continuousP}) retains the energy conservation property (\ref{energy}).

In (\ref{semiWG}), we chose a test function $v=\{v_0,v_b=0\}$ so
that $v_0=1$ on $K$ and $v_0=0$ elsewhere. After integrating over
the time period, we have
\begin{equation}\label{EC}
\int_{t-\Delta t}^{t+\Delta t}\int_K u_t\,dx\,dt+\int_{t-\Delta t}^{t+\Delta t}\int_{ K} a \nabla_d u\nabla_d v \,dx\,dt=\int_{t-\Delta t}^{t+\Delta t}\int_K f\,dx\,dt.
\end{equation}
Using the definition of operator $\bRh$ and of the weak gradient
$\nabla_d$ in (\ref{WGoperator}), we arrive at
$$
\int_{ K} a \nabla_d u\nabla_d v \,dx=\int_K\bRh(a\nabla_d u_h)\cdot
\nabla_d v\, dx=-\int_K\nabla\cdot \bRh(a\nabla_d
u_h)\,dx=-\int_{\partial K}\bRh(a\nabla_du_h)\cdot \bn\, ds.
$$

Then by substituting in (\ref{EC}), we have the energy conservation property
$$
\int_{t-\Delta t}^{t+\Delta t}\int_K u_t\,dx\,dt+\int_{t-\Delta
t}^{t+\Delta t} \int_{\partial K}-\bRh(a\nabla_du_h)\cdot \bn\,
ds\,dt=\int_{t-\Delta t}^{t+\Delta t}\int_K f\,dx\,dt,
$$
which provides a numerical flux given by
$$
q_h\cdot \bn=-\bRh(a\nabla_du_h)\cdot \bn.
$$
The numerical flux $q_h\cdot n$ is continuous across the edge of each element $T$, which can be verified by a selection of the test function $v=\{v_0,v_b\}$ so that $v_0\equiv0$ and $v_b$ arbitrary.

\section{Error Analysis}

In this section, we derive some error estimates for both continuous
and discrete time weak Galerkin finite element methods. The
difference between the weak Galerkin finite element approximation
$u_h$ and the $L^2$ projection $Q_h u$ of the exact solution $u$ is
measured. We first state a result concerning an approximation
property as follows.

\begin{lemma}\label{approx}
For $u\in H^{1+r}(\Omega)$ with $r>0$, we have
$$
\|\Pi_h(a\nabla u)-a \bRh(\nabla u)\|\leq Ch^r\|u\|_{1+r}.
$$
\end{lemma}
\begin{proof} Since from (\ref{identity1}) we have $\bRh(\nabla u)=\nabla_d(Q_h u)$, then
$$\|\Pi_h(a\nabla u)-a \bRh(\nabla u)\|=\|\Pi_h(a\nabla u)-a \nabla_d(Q_h u)\|.$$
Using the triangle inequality, the definition of $\Pi_h$ and the second condition {\bf P2} on the discrete weak space $S_h(j,l)$, we have
$$
\|\Pi_h(a\nabla u)-a \nabla_d(Q_h u)\|\leq\|\Pi_h(a\nabla u)-a \nabla u\|+\|a\nabla u-a \nabla_d(Q_h u)\|
\leq Ch^r\|u\|_{1+r}.
$$
\end{proof}

\subsection{Continuous Time Weak Galerkin Finite Element Method}

Our aim is to prove the following estimate for the error for the
semidiscrete solution.
\newtheorem{theorem}{Theorem}[section]
\begin{theorem}\label{theorem1}
Let $u\in H^{1+r}(\Omega)$ and $u_h$ be the solutions of
(\ref{continuousP}) and (\ref{semiWG}), respectively. Denote by
$e:=u_h-Q_h u$ the difference between the weak Galerkin
approximation and the $L^2$ projection of the exact solution $u$.
Then there exists a constant $C$ such that
$$
\|e(\cdot,t)\|^2+\int_0^t \alpha\|\nabla_d e\|^2\,ds
\leq \|e(\cdot, 0)\|^2+ Ch^{2r}\int_0^t \|u\|^2_{1+r}\,ds,
$$
and
 \begin{eqnarray*}
& &\int_0^t \|e_t\|^2\,ds+\frac{\alpha}{4}\|\nabla_d e(\cdot,t)\|^2 +
(1+\frac{\gamma}{2\alpha})\|e\|^2\\
&\leq& {\alpha}\|\nabla_d e(\cdot,0)\|^2 +
(1+\frac{\gamma}{2\alpha})\|e(\cdot,0)\|^2 \\
& &  + C h^{2r}\left(\|u(\cdot,0)\|^2_{1+r}+\|u\|^2_{1+r}+ \int_0^t
\|u\|^2_{1+r}\,ds+\int_0^t \|u_t\|^2_{1+r}\,ds\right).
\end{eqnarray*}
\end{theorem}
\begin{proof}
Let $v=\{v_0,v_b\}\in S^0_h(j,l)$ be the testing function. By
testing (\ref{heat}) against $v_0$, together with $\bRh(\nabla
u)=\nabla_d(Q_h u)$ for $u\in H^1$ and $(Q_0u_t,v_0)=(u_t,v_0)$, we
obtain
\begin{eqnarray*}
(f,v_0)&=&(u_t,v_0)+\sum_{T\in \T_h}(-\nabla\cdot(a\nabla u),v_0)_T\\
&=&(u_t,v_0)+(\Pi_h(a\nabla u),\nabla_d v)\\
&=& (Q_h u_t,v_0)+(\Pi_h(a\nabla u)-a \bRh(\nabla u),\nabla_d
v)+(a\nabla_d(Q_h u),\nabla_d v)
\end{eqnarray*}

Since the numerical solution also satisfies the heat equation, we have
$$
(f,v_0)=((u_{h})_{t}, v_0) + a(u_h,v).
$$
Combining the above two equations we obtain
\begin{equation}\label{errorEQ}
((u_h-Q_h u)_t, v_0)+ a(u_h-Q_h u,v)=(\Pi_h (a\nabla u)-a \bRh
(\nabla u),\nabla_d v),
\end{equation}
which shall be called the error equation for the WG for the heat
equation.

Let $e=u_h-Q_h u$ be the error. Use $v=e$ in the error equation, we obtain
$$
(e_t,e)+a(e,e)=(\Pi_h(a\nabla u)-a \bRh(\nabla u),\nabla_d e).
$$
By Cauchy-Schwarz inequality and the coercivity of the bilinear form, we have
$$
\frac12\frac{d}{dt}\|e\|^2+\alpha\|\nabla_d e\|^2\leq
\frac1{2\alpha}\|\Pi_h(a\nabla u)-a \bRh(\nabla
u)\|^2+\frac{\alpha}{2}\|\nabla_d e\|^2.
$$
It follows that
$$
\frac{d}{dt}\|e\|^2+\alpha\|\nabla_d e\|^2\leq
\frac1{\alpha}\|\Pi_h(a\nabla u)-a \bRh(\nabla u)\|^2,
$$
and hence, after integration,
\begin{equation}\label{errorL2}
\|e\|^2+\int_0^t \alpha\|\nabla_d e\|^2\,ds\leq \|e(\cdot,0)\|^2+
\frac1{\alpha}\int_0^t \|\Pi_h(a\nabla u)-a \bRh(\nabla u)\|^2\, dt.
\end{equation}
Then by Lemma \ref{approx}, we have the assertion.

In order to estimate $\nabla_d e$, we use the error equation with
$v=(u_h-Q_h u)_t=e_t$. We obtain
\begin{eqnarray*}
(e_t,e_t)+a(e,e_t)&=&(\Pi_h(a\nabla u)-a \bRh(\nabla u),\nabla_d e_t)\\
&=&\frac{d}{dt}(\Pi_h(a\nabla u)-a \bRh(\nabla u),\nabla_d
e)\\
& & -(\Pi_h(a\nabla u_t)-a \bRh(\nabla u_t),\nabla_d
e)\\
& & -(\Pi_h(a_t\nabla u)-a_t \bRh(\nabla u),\nabla_d e).
\end{eqnarray*}
By the Cauchy-Schwarz inequality, this shows that
\begin{eqnarray*}
\|e_t\|^2+\frac{1}{2}(\frac{d}{dt}a( e,e)-(a_t\nabla_d e,\nabla_d e))&\leq& \frac{d}{dt}(\Pi_h(a\nabla u)-a \bRh(\nabla u),\nabla_d e)\\
&+&\frac1{2\alpha} \|\Pi_h(a\nabla u_t)-a \bRh(\nabla u_t)\|^2+\frac{\alpha}{2} \|\nabla_d e\|^2\\
&+&\frac1{2\alpha} \|\Pi_h(a_t\nabla u)-a_t \bRh(\nabla
u)\|^2+\frac{\alpha}{2} \|\nabla_d e\|^2,
\end{eqnarray*}
i.e.,
\begin{eqnarray*}
\|e_t\|^2+\frac{1}{2}\frac{d}{dt}a( e,e)&\leq& \frac12(a_t\nabla_d e,\nabla_d e)+ \frac{d}{dt}(\Pi_h(a\nabla u)-a \bRh(\nabla u),\nabla_d e)\\ &+&\frac1{2\alpha} \|\Pi_h(a\nabla u_t)-a \bRh(\nabla u_t)\|^2+\frac{\alpha}{2} \|\nabla_d e\|^2\\
&+&\frac1{2\alpha} \|\Pi_h(a_t\nabla u)-a_t \bRh(\nabla
u)\|^2+\frac{\alpha}{2} \|\nabla_d e\|^2.
\end{eqnarray*}
Thus, integrating with respect to $t$ and together with the
coercivity and boundedness yields
\begin{eqnarray*}
& &\int_0^t \|e_t\|^2\,ds+\frac{\alpha}{2}\|\nabla_d e(\cdot,t)\|^2\\
&\leq& \frac{{\beta}}{2}\|\nabla_d e(\cdot,0)\|^2
 + (\Pi_h(a\nabla u(\cdot,t)-a \bRh(\nabla u(\cdot,t)),\nabla_d
 e(\cdot,t))\\
&-&(\Pi_h(a\nabla u(\cdot,0)-a \bRh(\nabla u(\cdot,0)),\nabla_d
e(\cdot,0))+\frac1{2\alpha}\int_0^t \|\Pi_h(a\nabla u_t)-a
\bRh(\nabla u_t)\|^2\,ds\\
&+&\frac1{2\alpha}\int_0^t \|\Pi_h(a_t\nabla u)-a_t \bRh(\nabla
u)\|^2\,ds+(\alpha+\frac{\gamma}{2}) \int_0^t\|\nabla_d e\|^2\,ds.
\end{eqnarray*}

By adding $(\alpha+\gamma/2)/\alpha=1+\frac{\gamma}{2\alpha}$ times inequality (\ref{errorL2}) to the above inequality, we arrive at
 \begin{eqnarray*}
& &\int_0^t \|e_t\|^2\,ds+\frac{\alpha}{2}\|\nabla_d e(\cdot,t)\|^2 +(1+\frac{\gamma}{2\alpha}) \|e\|^2\\
&\leq& \frac{{\beta}}{2}\|\nabla_d e(\cdot,0)\|^2 + (1+\frac{\gamma}{2\alpha})\|e(\cdot,0)\|^2\\
&+&(\Pi_h(a\nabla u(\cdot,t))-a \bRh(\nabla u(\cdot,t)),\nabla_d
e(\cdot,t))
{-}(\Pi_h(a\nabla u(\cdot,0))-a \bRh(\nabla u(\cdot,0)),\nabla_d e(\cdot,0)) \\
&+&\frac1{2\alpha}\int_0^t \|\Pi_h(a\nabla u_t)-a \bRh(\nabla
u_t)\|^2\,ds
+\frac1{2\alpha}\int_0^t \|\Pi_h(a_t\nabla u)-a_t \bRh(\nabla u)\|^2\,ds\\
&+&(\frac1{\alpha}+\frac{\gamma}{2\alpha^2})\int_0^t \|\Pi_h(a\nabla
u)-a \bRh(\nabla u)\|^2\, dt.
\end{eqnarray*}

Next, we use the Cauchy-Schwarz inequality to obtain
 \begin{eqnarray*}
& &\int_0^t \|e_t\|^2\,ds+\frac{\alpha}{4}\|\nabla_d e(\cdot,t)\|^2 + (1+\frac{\gamma}{2\alpha})\|e\|^2\\
&\leq& {{\beta}}\|\nabla_d e(\cdot,0)\|^2 + (1+\frac{\gamma}{2\alpha})\|e(\cdot,0)\|^2\\
&+&\frac1\alpha\|\Pi_h(a\nabla u(\cdot,t))-a \bRh(\nabla
u(\cdot,t))\|^2
+\frac1{2\alpha}\|\Pi_h(a\nabla u(\cdot,0))-a \bRh(\nabla u(\cdot,0))\|^2 \\
&+&\frac1{2\alpha}\int_0^t \|\Pi_h(a\nabla u_t)-a \bRh(\nabla
u_t)\|^2\,ds
+\frac1{2\alpha}\int_0^t \|\Pi_h(a_t\nabla u)-a_t \bRh(\nabla u)\|^2\,ds\\
&+&(\frac1{\alpha}+\frac{\gamma}{2\alpha^2})\int_0^t \|\Pi_h(a\nabla
u)-a \bRh(\nabla u)\|^2\, dt .
\end{eqnarray*}
Then by Lemma \ref{approx}, the proof is completed.
\end{proof}

\subsection{Discrete Time Weak Galerkin Finite Element Method}

We begin with the following lemma which provides a Poincar\'e-type
inequality with the weak gradient operator.

\begin{lemma}\label{poincare}
Assume that {$\phi=\{\phi_0,\phi_b\}\in S^0_h (j,j)$}, then there
exists a constant $C$ such that
$$
\|\phi\|\leq C \|\nabla_d \phi\|,
$$
where $\nabla_d \phi\in V(T,r=j+1)=[P_{j}(T)]^2+\hat{P}_j(T){\bf x}$.
\end{lemma}

\begin{proof}
Let $\bar{\phi}_0$ be a piecewise constant function with the cell
average of $\phi$ on each element $T$. Let $\phi_I\in H^1_0(\Omega)$
be a continuous piecewise polynomial with vanishing boundary value
lifted from $\phi$ as follows. Let $G_j(T)$ be the set of all
Lagrangian nodal points for $P_{j+1}(T)$. At all internal Lagrangian
nodal points $x_i\in G_j(T)$, we set $\phi_I(x_i)=\bar{\phi}_0$. At
boundary Lagrangian points $x_i\in G_j(T)\cap\partial T$, we let
$\phi_I(x_i)$ be the trace of $\bar{\phi}_0$ from either side of the
boundary. At global Lagrangian points $x_i\in
\partial T\cap\partial \Omega$, we set $\phi_I(x_i)=0$. Let
$\jump{\phi_0}_e$ denote the jump of $\phi_0$ on the edge $e$; i.e.,
\begin{equation}\label{jw.05}
\jump{\phi_0}_e=\phi_0|_{T_1}-\phi_0|_{T_2}
=(\phi_0|_{T_1}-\phi_b)-(\phi_0|_{T_2}-\phi_b).
\end{equation}
By the classical Poincar\'e inequality for $\phi_I$, we have
\begin{eqnarray}\label{jw.07}
\|\phi\|&\leq& \|\phi-\phi_I\|+\|\phi_I\|\\
&\leq&\left(\sum_T \|\phi-\phi_I\|^2_{T}\right)^{\frac12}+
C\|\nabla\phi_I\|.\nonumber
\end{eqnarray}

From Lemma 4.3 in \cite{JW_WG_stoke}, we have
\begin{equation}\label{jw.06}
\|\phi-\phi_I\|^2_T \leq \sum_{T'\in  \mathcal{T}(T)} h_{T'}^2
\|\nabla \phi_0\|_{T'}^2 +\sum_{e\in  \large{\varepsilon}(T)}h_e
\|\jump{\phi_0}\|_e^2,\quad \forall T\in \mathcal{T}_h,
\end{equation}
where $ \mathcal{T}(T)$ denotes the set of all triangles in
$\mathcal{T}$ having a nonempty intersect with $T$, including $T$
itself, and  $\large{\varepsilon}(T)$ denotes the set of all edges
having a nonempty intersection with $T$. Note the elementary fact
that
\begin{equation}\label{jw.02}
\|\nabla\phi_I\|^2\leq C\sum_T |\phi_I(x_i)-\phi_I(x_k)|^2,
\end{equation}
where $x_i$ and $x_k$ run through all the Lagrangian nodal points on
$T$. By construction, $\phi_I(x_i)$ is the cell average of $\phi_0$
on either the element $T$ or an adjacent element $T_*$ which shares
with $T$ a common edge or a vertex point. Thus,
$\phi_I(x_i)-\phi_I(x_k)$ is either zero or the difference of the
cell average of $\phi_0$ on two adjacent elements $T$ and $T_*$. For
the later case, assume that $x_e$ is a point shared by $T$ and
$T_*$. It is not hard to see that
$$
\left|\bar\phi_0|_T - \phi_0(x_e)\right|^2 \leq C \int_T
|\nabla\phi_0|^2 dx.
$$
Thus, we have
\begin{eqnarray*}\label{jw.01}
|\phi_I(x_i)-\phi_I(x_k)|^2&=&|\bar\phi_0|_T - \bar\phi_0|_{T_*}|^2\\
&=& \left|\bar\phi_0|_T-\phi_0|_T(x_e)+ \jump{\phi_0}(x_e) +
\phi_0|_{T_*}(x_e)- \bar\phi_0|_{T_*}\right|^2\nonumber\\
&\le& C\sum_{e\subset\partial T} h_e^{-1}  \|\jump{\phi_0}\|_e^2+C
\int_{T\cup T_*} |\nabla\phi_0|^2 dx.\nonumber
\end{eqnarray*}
Substituting the above estimate into (\ref{jw.02}) yields
\begin{equation}\label{jw.03}
\|\nabla\phi_I\|^2\leq C\sum_T \left(\|\nabla \phi_0\|_T^2+h^{-1}
\|\jump{\phi_0}\|_{\partial T}^2\right).
\end{equation}
By combining (\ref{jw.07}) with (\ref{jw.06}) and (\ref{jw.03}) we
obtain
\begin{equation}\label{bound1}
\|\phi\|^2\leq C\sum_T \left(\|\nabla \phi_0\|_T^2+h^{-1}
\|\phi_0-\phi_b\|_{\partial T}^2\right),
\end{equation}
where (\ref{jw.05}) has been applied to estimate the jump of
$\phi_0$ on each edge.

Next, we want to bound the two terms on the right hand side of
(\ref{bound1}) by $\|\nabla_d \phi\|$. Let us recall that the weak
gradient $\nabla_d\phi$ is defined by
$$
\int_T\nabla_d \phi\cdot {\bf q} dx=-\int_T\phi_0 \nabla\cdot {\bf
q} dx +\int_{\partial T} \phi_b {\bf q}\cdot \bn ds, \quad \forall
{\bf q}\in V(T,j+1),
$$
and by using integration by parts, we have
\begin{equation}\label{WG2}
\int_T\nabla_d \phi\cdot {\bf q} dx=\int_T\nabla\phi_0 \cdot {\bf q}
dx-\int_{\partial T} (\phi_0-\phi_b) {\bf q}\cdot \bn ds, \quad
\forall {\bf q}\in V(T,j+1).
\end{equation}
In order to bound $\|\nabla \phi_0\|$ and $ \|\jump{\phi_0}\|_e$ by
$\|\nabla_d \phi\|$, we let ${\bf q}$ in (\ref{WG2}) satisfy the
following
\begin{eqnarray*}
\int_T {\bf q}\cdot {\bf r} \,dx&=&\int_T \nabla \phi_0\cdot {\bf r}
\,dx\quad\forall {\bf r}\in [P_j(T)]^2,\\
\int_{\partial T} {\bf q}\cdot \bn \mu\,ds&=& -h^{-1}\int_{\partial
T} ( \phi_0-\phi_b)\mu\,ds\quad\forall \mu\in P_{j+1}[\partial T],
\end{eqnarray*}
where by Lemma 5.1 in \cite{stoke}, ${\bf q}$ and $\phi$ have the
norm equivalence of
\begin{equation}\label{norm}
\|{\bf q}\|_{L^2(T)}\approx\|\nabla
\phi_0\|_{L^2(T)}+h^{-\frac12}\|\phi_0-\phi_b\|_{L^2(\partial T)}.
\end{equation}

Then from (\ref{WG2}), we have
 \begin{eqnarray*}
 \int_T\nabla_d \phi\cdot {\bf q} dx&=&\int_T\nabla\phi_0 \cdot {\bf q} dx-\int_{\partial T} (\phi_0-\phi_b) {\bf q}\cdot \bn ds\\
 &=&\int_T \nabla \phi_0\cdot \nabla \phi_0 \,dx+h^{-1}\int_{\partial T} ( \phi_0-\phi_b)^2\,ds\\
&=& \|\nabla \phi_0\|_{T}^2+h^{-1}\| \phi_0-\phi_b\|^2_{\partial T}.
 \end{eqnarray*}
Using (\ref{WG2}) and (\ref{norm}), we have
\begin{eqnarray*}
\|\nabla_d \phi\|_T\geq \frac1{\|{\bf q}\|}(\nabla_d \phi, {\bf q})_T&=&\frac1{\|{\bf q}\|}(\| \nabla\phi_0\|_{T}^2+h^{-1}\| \phi_0-\phi_b\|^2_{\partial T})\\
&\geq& C (\| \nabla\phi_0\|_{T}^2+h^{-1}\|
\phi_0-\phi_b\|^2_{\partial T}),
\end{eqnarray*}
then together with (\ref{bound1}), we have the assertion.
\end{proof}

With the results established in Lemma \ref{poincare}, we are ready
to derive an error estimate for the discrete time weak Galerkin
approximation $u_h$ as the following theorem.

\begin{theorem}\label{theorem2}
Let $u\in H^{1+r}(\Omega)$ and $U^n$ be the solutions of (\ref{continuousP}) and (\ref{discreteWG}), respectively. Denote by $e^n:=U^n-Q_h u(t_n)$ the difference between the backward Euler weak Galerkin approximation and the $L^2$ projection of the exact solution $u$. Then there exists a constant $C$ such that
$$
\|e^n\|^2+\sum_{j=1}^n\alpha k\|\nabla_d e^j\|^2\leq \|e^0\|^2+ C(h^{2r}\|u\|^2_{1+r}+k^2\int_0^{t_n} \|u_{tt}\|^2\,ds),\quad \mbox{ for } n\geq 0,
$$
and
$$
\|\nabla_d e^n\|^2\leq C\left(\|e^0\|^2+\|\nabla_d e^0\|^2 +h^{2r}(\|u\|_{1+r}^2+\|u_t\|^2_{1+r})+k^2\int_0^{t_n}\|u_{tt}\|^2ds\right),
$$
where $\|u\|_{1+r}=\displaystyle\max_{j=1,\cdots,n}\{\|u(t_j)\|_{1+r}\}$ and $\|u_t\|_{1+r}=\displaystyle\max_{j=1,\cdots,n}\{\|u_t(t_j)\|_{1+r}\}$.
\end{theorem}
\begin{proof}
A calculation corresponding to the error equation (\ref{errorEQ})
yields
\begin{eqnarray*}
&& \quad (\bar{\partial} (U^n-Q_h u(t_n)),v_0)+a(U^n-Q_h
u(t_n)),v)\\
&&=(u_t(t_n)-\bar{\partial}u(t_n),v_0)+(\Pi_h (a\nabla u(t_n))-a
\bRh (\nabla u(t_n)),\nabla_d v),
\end{eqnarray*}
i.e.,
\begin{equation}\label{DerrorEQ}
(\bar{\partial}
e^n,v_0)+a(e^n,v)=(u_t(t_n)-\bar{\partial}u(t_n),v_0)+(\Pi_h
(a\nabla u(t_n))-a \bRh (\nabla u(t_n)),\nabla_d v).
\end{equation}

Let $w_1^n=u_t(t_n)-\bar{\partial}u(t_n)$, and $w_2^n=\Pi_h (a\nabla
u(t_n))-a \bRh (\nabla u(t_n))$, and choosing $v=e^n$ in
(\ref{DerrorEQ}), we have
$$
(\bar{\partial} e^n,e^n)+a(e^n,e^n)=(w_1^n,e^n)+(w_2^n,\nabla_d e^n).
$$
By coercivity of the bilinear form and Cauchy-Schwarz inequality, we obtain
$$
\|e^n\|^2-(e^{n-1},e^n)+\alpha k\|\nabla_d e^n\|^2\leq k\|w_1^n\|\|e^n\|+k\|w_2^n\|\|\nabla_d e^n\|,
$$
i.e.,
$$
\|e^n\|^2+\alpha k\|\nabla_d e^n\|^2\leq \frac12\|e^{n-1}\|^2+\frac12\|e^{n}\|^2+k\|w_1^n\|\|e^n\|+k\|w_2^n\|\|\nabla_d e^n\|.
$$
By the Poincar\'e inequality of Lemma \ref{poincare}, we have
$$
\frac12\|e^{n}\|^2+\alpha k\|\nabla_d e^n\|^2\leq\frac12\|e^{n-1}\|^2+k(C\|w_1^n\|+\|w_2^n\|)(\|\nabla_d e^n\|).
$$
Then by triangle inequality, it follows
$$
\frac12\|e^n\|^2+\frac{\alpha k}2\|\nabla_d e^n\|^2\leq\frac12\|e^{n-1}\|^2+\frac{k}{2\alpha}(C\|w_1^n\|+\|w_2^n\|)^2.
$$
so that
\begin{eqnarray*}
\|e^n\|^2+\alpha k\|\nabla_d e^n\|^2\leq \|e^{n-1}\|^2+\frac{Ck}{\alpha}\|w_1^n\|^2+\frac{k}{\alpha}\|w_2^n\|^2.
\end{eqnarray*}
and, by repeated application,
\begin{equation}\label{Derror}
\|e^n\|^2+\sum_{j=1}^n\alpha k\|\nabla_d e^j\|^2\leq \|e^0\|^2+\frac{Ck}{\alpha}\sum_{j=1}^n \|w_1^j\|^2+\frac k\alpha\sum_{j=1}^n \|w_2^j\|^2.
\end{equation}

We write
\begin{equation}\label{w1integral}
kw_1^j=ku_t(t_j)-(u(t_j)-u(t_{j-1}))=\int_{t_{j-1}}^{t_j} (s-t_{j-1})u_{tt}(s)\,ds,
\end{equation}
i.e.,
$$
w_1^j=u_t(t_j)-\frac{u(t_j)-u(t_{j-1})}{k}=\frac1k\int_{t_{j-1}}^{t_j} (s-t_{j-1})u_{tt}(s)\,ds,
$$
so that
\begin{eqnarray}\label{w1}
\|w_1^j\|^2&=&\int_\Omega  \{\frac1k\int_{t_{j-1}}^{t_j} (s-t_{j-1})u_{tt}(s)\,ds\}^2\,dx\nonumber\\
&\leq&\frac1{k^2} \int_\Omega \int_{t_{j-1}}^{t_j}(s-t_{j-1})^2\,ds \int_{t_{j-1}}^{t_j} u^2_{tt}(s)\,ds\,\,dx\\
&\leq& Ck \int_{t_{j-1}}^{t_j} \|u_{tt}\|^2\,ds.\nonumber
\end{eqnarray}
Substitute (\ref{w1}) into (\ref{Derror}) and together with Lemma
\ref{approx}, we have the error estimate for $\|e^n\|$.

In order to show an estimate for $\nabla_d e^n$ we may choose
instead $v=\bar{\partial} e^n$ in error equation (\ref{Derror}) to
obtain the following identity
$$
(\bar{\partial} e^n,\bar{\partial} e^n)+a(e^n,\bar{\partial} e^n)=(w^n_1,\bar{\partial} e^n)+(w^n_2,\nabla_d \bar{\partial} e^n).
$$
The second term on the right hand side can be written as
$$
(w^n_2,\nabla_d \bar{\partial} e^n)=\bar{\partial} (w^n_2,\nabla_d e^n)-((w_2^n)_t-\bar{\partial} w^n_2,\nabla_d e^{n-1})+((w_2^n)_t,\nabla_d e^{n-1}).
$$
Then the error equation becomes
\begin{eqnarray*}
k\|\bar{\partial} e^n\|^2+(a\nabla_d e^n,\nabla_d e^n)&=&(a\nabla_d e^n,\nabla_d e^{n-1})+k(w_1,\bar{\partial} e^n)\\
&+&k\bar{\partial}(w^n_2,\nabla_d e^n)-k((w_2^n)_t-\bar{\partial} w^n_2,\nabla_d e^{n-1})+k((w_2^n)_t,\nabla_d e^{n-1}).
\end{eqnarray*}
By triangle inequality, we have
\begin{eqnarray*}
k\|\bar{\partial} e^n\|^2+(a\nabla_d e^n,\nabla_d e^n)&\leq& \frac12(a\nabla_d e^n,\nabla_d e^n)+\frac12(a\nabla_d e^{n-1},\nabla_d e^{n-1})\\
&+&\frac{k}4\|w_1^n\|^2+k\|\bar{\partial} e^n\|^2+k\bar{\partial}(w^n_2,\nabla_d e^n)\\
&+&\frac{k}2\|(w_2^n)_t-\bar{\partial} w^n_2\|^2+\frac{k}2\|\nabla_d e^{n-1}\|^2\\
&+&\frac{k}{2}\|(w_2^n)_t\|^2+\frac{k}{2}\|\nabla_d e^{n-1}\|^2,
\end{eqnarray*}
and, after cancelation and by repeated application,
\begin{eqnarray*}
\frac12(a\nabla_d e^n,\nabla_d e^n)&\leq& \frac12(a\nabla_d e^{0},\nabla_d e^{0})\\
&+&\frac{k}4\sum_{j=1}^n\|w_1^j\|^2+(w^n_2,\nabla_d e^n)-(w^0_2,\nabla_d e^0)\\
&+&\frac{k}2\sum_{j=1}^n\|(w_2^n)_t-\bar{\partial} w^j_2\|^2+\frac{k}{2}\sum_{j=1}^n\|(w_2^n)_t\|^2+k\sum_{j=1}^n\|\nabla_d e^{j-1}\|^2,
\end{eqnarray*}
which is
\begin{eqnarray}\label{graderror}
\frac{\alpha}2\|\nabla_d e^n\|^2&\leq& \frac{\beta}2\|\nabla_d e^{0}\|^2\nonumber\\
&+&\frac{k}4\sum_{j=1}^n\|w_1^j\|^2+\frac1{\alpha}\|w^n_2\|^2+\frac{\alpha}4\|\nabla_d e^n\|^2+\frac1{2\beta}\|w^0_2\|^2+\frac{\beta}2\|\nabla_d e^0\|^2\\
&+&\frac{k}2\sum_{j=1}^n\|(w_2^n)_t-\bar{\partial} w^j_2\|^2+\frac{k}{2}\sum_{j=1}^n\|(w_2^n)_t\|^2+k\sum_{j=1}^n\|\nabla_d e^{j-1}\|^2,\nonumber
\end{eqnarray}
followed by triangle inequality and boundness and coercivity of the bilinear form.
By the similar process as in (\ref{w1}) and Lemma \ref{approx}, we have
$$
\|(w_2^n)_t-\bar{\partial} w^j_2\|^2\leq Ck\int_{t_{j-1}}^{t_j}\|(w_2)_{tt}\|^2\,ds\leq C kh^{2r}\int_{t_{j-1}}^{t_j}\|u_{tt}\|^2_{1+r}\,ds.
$$

Then by substituting (\ref{w1}), (\ref{Derror}) and the above
inequality into (\ref{graderror}), we have the error estimate for
$\|\nabla_d e^n\|$ as the following
$$
\|\nabla_d e^n\|^2\leq C\left(\|e^0\|^2+\|\nabla_d e^0\|^2 +h^{2r}(\|u\|_{1+r}^2+\|u_t\|^2_{1+r})+k^2\int_0^{t_n}\|u_{tt}\|^2ds\right),
$$
which completes the proof.
\end{proof}

\subsection{Optimal Order of Error Estimation in $L^2$}
To get an optimal order of error estimate in $L^2$, the idea,
similar to Wheeler's projection as in \cite{Wheeler} and
\cite{VThomee}, is used where an elliptic projection $E_h$ onto the
discrete weak space $S_h(j,l)$ is defined as the following: Find
$E_h v\in S_h(j,l)$ such that $E_hv$ is the $L^2$ projection of the
trace of $v$ on the boundary $\partial\Omega$ and
\begin{equation}\label{ellipticOP}
(a\nabla_d E_h v,\nabla_d \chi)=(-\nabla\cdot(a\nabla v),\chi),\quad
\forall\chi\in S_h^0(j,l).
\end{equation}
In view of the weak formulation of the elliptic problem,
\begin{eqnarray}\label{ellipticP}
-\nabla\cdot(a\nabla v)= F\quad \mbox{ in } \Omega,\\
v=g\quad \mbox{on }\partial \Omega,\nonumber
\end{eqnarray}
this definition may be expressed by saying that $E_h v$ is the weak Galerkin finite element approximation of the solution of the corresponding elliptic problem with exact solution $v$. By the error estimate results in \cite{JW_WG}, we have the following error estimate for $E_h v$.
\begin{lemma}\label{elliptic_error}
Assume that problem (\ref{ellipticP}) has the $H^{1+s}$ regularity $(s\in (0,1])$. Let $v\in H^{1+r}$ be the exact solution of (\ref{ellipticP}), and $E_h v$ be a weak Galerkin approximation of $v$ defined in (\ref{ellipticOP}). Let $Q_h v=\{Q_0 v, Q_b v\}$ be the $L^2$ projection of $v$ in the corresponding finite element space. Then there exists a constant $C$ such that
\begin{equation}\label{ellipticEE}
\|Q_0 v-E_h v\|\leq C(h^{s+1}\|F-Q_0 F\|+h^{r+s}\|v\|_{r+1}),
\end{equation}
and
\begin{equation}\label{ellipticEE2}
\|\nabla_d (Q_h v-E_h v)\|\leq Ch^r\|v\|_{r+1}.
\end{equation}
\end{lemma}

Throughout this section the error in the parabolic problem
(\ref{continuousP}) is written as a sum of two terms,
\begin{equation}\label{twoterm}
u_h(t)-Q_h u(t)= \theta(t)+\rho(t),\quad \mbox{ where } \theta=u_h-E_h u,\quad \rho=E_h u-Q_h u,
\end{equation}
which will be bounded separately. Notice that the second term is the error in an elliptic problem and then can be handled by applying the results in Lemma \ref{elliptic_error}. Then our main goal here is to bound the first term $\theta$.

Following the above strategy, the error estimate for continuous time weak Galerkin finite element method in $L^2$ and $H^1$ are provided in the next two theorems.

\begin{theorem}\label{RiezeL2}
Under the assumption of Theorem \ref{theorem1} and the assumption that the corresponding elliptic problem has the $H^{1+s}$ regularity $(s\in (0,1])$, there exists a constant $C$ such that
\begin{eqnarray*}
\|u_h(t)-Q_h u(t)\|&\leq& \|u_h(0)-Q_h u(0)\|+Ch^{r+s}(\|\psi\|_{r+1}+\int_0^t\|u_t\|_{r+1}\,ds)\\
&+&Ch^{s+1}(\|f(0)-Q_0f(0)\|+\|u_t(0)-Q_0u_t(0)\|)\\
&+&Ch^{s+1}\left\{ \int_0^t (\|f_t-Q_0f_t\|+\|u_{tt}-Q_0u_{tt}\|)\,ds\right\}.
\end{eqnarray*}
\end{theorem}
\begin{proof}
We write the error according to (\ref{twoterm}) and obtain the error bound for $\rho$ easily by Lemma \ref{elliptic_error} as the following
\begin{equation}\label{rhobound}
\|\rho\|\leq C(h^{s+1}(\|f-Q_0f\|+\|u_t-Q_0u_t\|)+h^{r+s}(\|\psi\|_{r+1}+\int_0^t\|u_t\|_{r+1}\,ds)).
\end{equation}
In order to estimate $\theta$, we note that by our definitions
\begin{eqnarray*}
(\theta_t,\chi)+a(\theta,\chi)&=&(u_{h,t},\chi)+a(u_h,\chi)-(E_h u_t,\chi)-a(E_h u, \chi)\\
&=&(f,\chi)-(E_h u_t,\chi)-a(E_h u, \chi)\\
&=&(f,\chi)+(\nabla\cdot(a\nabla u),\chi)-(E_h u_t,\chi)\\
&=&(u_t,\chi)-(E_h u_t,\chi)\\
&=&(Q_h u_t,\chi)-(E_h u_t,\chi)\\
&=&-(\rho_t,\chi),
\end{eqnarray*}
which is
\begin{equation}\label{thetaeq}
(\theta_t,\chi)+a(\theta,\chi)=-(\rho_t,\chi),\quad \forall \chi\in
S_h^0(j,l),\,\: t>0,
\end{equation}
where we have used the fact that the operator $E_h$ commutes with
time differentiation. Since $\theta\in S_h^0(j,l)$, we may choose
$\chi=\theta$ in (\ref{thetaeq}) and obtain
$$
(\theta_t,\theta)+a(\theta,\theta)=-(\rho_t,\theta),\quad t>0.
$$
Since
$$
a(\theta,\theta)\geq \alpha\|\nabla_d\theta\|^2>0,
$$
we have
$$
\frac12\frac{d}{dt}\|\theta\|^2=\|\theta\|\frac{d}{dt} \|\theta\|\leq \|\rho_t\|\|\theta\|,
$$
and hence
$$
\|\theta(t)\|\leq \|\theta(0)\|+\int_0^t\|\rho_t\|\,ds.
$$
Using Lemma \ref{elliptic_error}, we find
\begin{eqnarray}\label{theta0}
\|\theta(0)\|&=&\|u_h(0)-E_h u(0)\|\\
&\le&\|u_h(0)-Q_h u(0)\|+\|E_h u(0)-Q_h u(0)\| \nonumber \\
&\leq&\|u_h(0)-Q_h u(0)\|\nonumber\\
& & +C[h^{s+1}(\|f(0)-Q_0
f(0)\|+\|u_t(0)-Q_0u_t(0)\|)+h^{r+s}\|\psi\|_{r+1}],\nonumber
\end{eqnarray}
and since
\begin{equation}\label{rhot}
\|\rho_t\|=\|E_h u_t-Q_h u_t\|\leq C[h^{s+1}(\|f_t-Q_0 f_t\|+\|u_{tt}-Q_0u_{tt}\|)+h^{r+s}\|u_{t}\|_{r+1}],
\end{equation}
the desired bound for $\|\theta(t)\|$ now follows.
\end{proof}

\begin{theorem}\label{RiezeH1}
Under the assumption of Theorem \ref{RiezeL2} and the assumption that the coefficient matrix $a$ in (\ref{continuousP}) is independent of time $t$, there exists a constant $C$ such that
\begin{eqnarray*}
\|\nabla_d(u_h(t)-Q_h u(t))\|^2&\leq&4\beta \|\nabla_d(u_h(0)-Q_h u(0))\|^2+Ch^{2r}(\|\psi\|_{r+1}^2+\|u\|^2_{r+1})\\
&+&Ch^{2(s+1)}\int_0^t (\|f_t-Q_0 f_t\|+\|u_{tt}-Q_0u_{tt}\|)^2\,ds+Ch^{2(r+s)}\int_0^t \|u_{t}\|^2_{r+1}\,ds.
\end{eqnarray*}
\end{theorem}
\begin{proof}
As in the proof of Theorem \ref{RiezeL2}, we write the error in the form (\ref{twoterm}). Here by Lemma \ref{elliptic_error},
\begin{equation}\label{gradrho}
\|\nabla_d \rho(t)\|\leq Ch^{r}\|u\|_{r+1}.
\end{equation}
In order to estimate $\nabla_d\theta$, we may choose $\chi=\theta_t$
in the equation (\ref{thetaeq}) for $\theta$. We obtain
$$
(\theta_t,\theta_t)+a(\theta,\theta_t)=-(\rho_t,\theta_t).
$$
Since the coefficient matrix $a$ in the bilinear form
$a(\cdot,\cdot)$ is independent of time $t$, we have
$$
\|\theta_t\|^2+\frac12\frac{d}{dt}(a\nabla_d\theta,\nabla_d\theta)=-(\rho_t,\theta_t)\leq\frac12\|\rho_t\|^2+\frac12\|\theta_t\|^2,
$$
so that
$$
\frac{d}{dt}(a\nabla_d\theta,\nabla_d\theta)\leq\|\rho_t\|^2.
$$
Then by integrating with respect to time $t$ and using the coercivity and boundedness of the bilinear form, we obtain
\begin{eqnarray*}
\alpha\|\nabla_d \theta\|^2&\leq& (a\nabla_d\theta,\nabla_d\theta)\leq (a\nabla_d\theta(0),\nabla_d\theta(0))+\int_0^t \|\rho_t\|^2\, ds
\leq \beta\|\nabla_d\theta(0)\|^2+\int_0^t \|\rho_t\|^2\, ds\\
&\leq& \beta(\|\nabla_d(u_h(0)-Q_h u(0))\|+\|\nabla_d(E_h u(0)-Q_h u(0))\|)^2+\int_0^t \|\rho_t\|^2\, ds.
\end{eqnarray*}
Hence, in view of Lemma \ref{elliptic_error} and (\ref{rhot}), we have
\begin{eqnarray*}
\|\nabla_d\theta(t)\|^2&\leq& 2\beta \|\nabla_d(u_h(0)-Q_h u(0))\|^2+Ch^{2r}\|\psi\|_{r+1}^2\\
& &+Ch^{2(s+1)}\int_0^t (\|f_t-Q_0
f_t\|+\|u_{tt}-Q_0u_{tt}\|)^2\,ds+Ch^{2(r+s)}\int_0^t
\|u_{t}\|^2_{r+1}\,ds,
\end{eqnarray*}
which completes the proof.
\end{proof}

Next, we derive an error estimate for the backward Euler weak
Galerkin method.
\begin{theorem}\label{RiezeDL2}
Let $u\in H^{1+r}(\Omega)$ and $U^n$ be the solutions of
(\ref{continuousP}) and (\ref{discreteWG}), respectively. And let
$Q_h u$ be the $L^2$ projection of the exact solution $u$. Then
there exists a constant $C$ such that
\begin{eqnarray*}
\|U^n-Q_h u(t_n)\|&\leq& \|U^0-Q_h u(0)\|+Ch^{r+s}(\|\psi\|_{r+1}+\int_0^{t_n}\|u_t\|_{r+1}\,ds)\\
&+&Ch^{s+1}(\|f(0)-Q_0f(0)\|+\|u_t(0)-Q_0u_t(0)\|)\\
&+&Ch^{s+1}(\|f(t_n)-Q_0f(t_n)\|+\|u_t(t_n)-Q_0u_t(t_n)\|)\\
&+&Ch^{s+1}\left\{ \int_0^{t_n} (\|f_t-Q_0f_t\|+\|u_{tt}-Q_0u_{tt}\|)\,ds\right\}\\
&+&Ck\int_0^{t_n} \|u_{tt}\|\,ds.
\end{eqnarray*}
\end{theorem}
\begin{proof}
In analogy  with (\ref{twoterm}), we write
$$
U^n-Q_h u(t_n)=(U^n-E_h u(t_n))+(E_h u(t_n)-Q_h u(t_n))=\theta^n+\rho^n,
$$
where $\rho^n=\rho(t_n)$ is bounded as shown in (\ref{rhobound}). In order to bound $\theta^n$, we use
\begin{eqnarray*}
(\bar{\partial} \theta^n, \chi)+a(\theta^n,\chi)&=&(\bar{\partial} U^n, \chi)+a(U^n,\chi)-(\bar{\partial} E_h u(t_n), \chi)-a(E_h u(t_n),\chi)\\
&=&(f(t_n),\chi)-(\bar{\partial} E_h u(t_n), \chi)-a(E_h u(t_n),\chi)\\
&=&(f(t_n),\chi)+(\nabla\cdot(a\nabla u(t_n)),\chi)-(\bar{\partial} E_h u(t_n), \chi)\\
&=&(u_t(t_n),\chi)-(\bar{\partial} E_h u(t_n), \chi)\\
&=&(u_t(t_n)-\bar{\partial} u(t_n),\chi)+(\bar{\partial} u(t_n)-\bar{\partial} E_h u(t_n), \chi),
\end{eqnarray*}
i.e.,
\begin{equation}\label{thetabar}
(\bar{\partial} \theta^n, \chi)+a(\theta^n,\chi)=(w^n,\chi),
\end{equation}
where
$$
w^n=(u_t(t_n)-\bar{\partial} u(t_n))+(\bar{\partial} u(t_n)-\bar{\partial} E_h u(t_n))=w^n_1+w^n_3.
$$
By choosing $\chi=\theta^n$ in (\ref{thetabar}) and the coercivity of the bilinear form, we have
$$
(\bar{\partial} \theta^n, \theta^n)\leq\|w^n\|\|\theta^n\|,
$$
or
$$
\|\theta^n\|^2-(\theta^{n-1}, \theta^n)\leq k \|w^n\|\|\theta^n\|,
$$
so that
$$
\|\theta^n\|\leq \|\theta^{n-1}\|+k\|w^n\|,
$$
and, by repeated application, it follows
$$
\|\theta^n\|\leq \|\theta^0\|+k\sum^n_{j=1}\|w^j\|\leq  \|\theta^0\|+k\sum^n_{j=1}\|w_1^j\|+k\sum^n_{j=1}\|w_3^j\|.
$$
As in (\ref{theta0}), $\theta^0=\theta(0)$ is bounded as desired. By using the representation in (\ref{w1integral}), we obtain
$$
k\sum^n_{j=1}\|w_1^j\|\leq \sum^n_{j=1}\|\int_{t_{j-1}}^{t_j} (s-t_{j-1})u_{tt}(s)\,ds\|\leq k\int_0^{t_n}\|u_{tt}\|\,ds.
$$
We write
\begin{equation}\label{w3}
w_3^j=\bar{\partial} u(t_n)-E_h \bar{\partial} u(t_n)=(E_h-I)k^{-1}\int_{t_{j-1}}^{t_j} u_t\,ds=k^{-1} \int_{t_{j-1}}^{t_j} (E_h-I) u_t\,ds,
\end{equation}
and, by (\ref{rhot}) we have
\begin{eqnarray*}
k\sum^n_{j=1}\|w_3^j\|&\leq& \sum^n_{j=1}\int_{t_{j-1}}^{t_j} C[h^{s+1}(\|f_t-Q_0 f_t\|+\|u_{tt}-Q_0u_{tt}\|)+h^{r+s}\|u_{t}\|_{r+1}]\, ds\\
&=&C[h^{s+1}\int_{0}^{t_n} (\|f_t-Q_0 f_t\|+\|u_{tt}-Q_0u_{tt}\|)\,ds+h^{r+s}\int_{0}^{t_n} \|u_{t}\|_{r+1}\,ds],
\end{eqnarray*}
Thus, together with the estimate of $\rho$ in (\ref{rhobound}), we
have the assertion.
\end{proof}

\begin{theorem}
Under the assumption of Theorem \ref{RiezeDL2}, and the assumption
that the coefficient matrix $a$ is independent of time $t$, there
exists a constant $C$ such that
\begin{eqnarray*}
\|\nabla_d(U^n-Q_h u(t_n))\|^2&\leq&2\|\nabla_d(U^0-Q_h u(0))\|^2+Ch^{2r}(\|\psi\|_{r+1}^2+\|u\|^2_{r+1})\\
&+&Ch^{2(s+1)}\int_0^{t_n} (\|f_t-Q_0 f_t\|+\|u_{tt}-Q_0u_{tt}\|)^2\,ds+Ch^{2(r+s)}\int_0^{t_n} \|u_{t}\|^2_{r+1}\,ds\\
&+&Ck^2\int_0^{t_n} \|u_{tt}\|^2\,ds.
\end{eqnarray*}
\end{theorem}
\begin{proof}
It is sufficient to estimate $\nabla_d \theta^n$. To this end, we
choose $\chi=\bar{\partial}\theta^n$ in (\ref{thetabar}), and it is
easily seen that
$$
(\bar{\partial} \theta^n, \bar{\partial}
\theta^n)+a(\theta^n,\bar{\partial} \theta^n) =\|\bar{\partial}
\theta^n\|^2+\frac12\bar{\partial}a(\theta^n,\theta^n)+\frac{k}2a(\bar{\partial}\theta^n,
\bar{\partial}\theta^n)=(w^n,\bar{\partial}\theta^n),
$$
so that
$$
\bar{\partial}a(\theta^n,\theta^n)\leq \|w^n\|^2.
$$
By repeating the application, we have
\begin{eqnarray*}
a(\theta^n, \theta^n)&\leq& a(\theta^0,\theta^0)+k\sum_{j=0}^n\|w^j\|^2 \\
&\leq& a(\theta^0,\theta^0)+2k\sum_{j=0}^n
\|w_1^j\|^2+2k\sum_{j=0}^n \|w_3^j\|^2.
\end{eqnarray*}
As in (\ref{w3}), we obtain
$$
k\|w_3^j\|^2=k\int_\Omega (k^{-1} \int_{t_{j-1}}^{t_j}  \rho_t\,ds)^2\, dx\leq \int_\Omega ( \int_{t_{j-1}}^{t_j}  \rho_t^2\,ds)\, dx\leq \int_{t_{j-1}}^{t_j}  \|\rho_t\|^2\,ds.
$$
Together with (\ref{w1}), (\ref{rhot}) and (\ref{gradrho}), we have
the assertion.
\end{proof}

\section{Numerical Experiments}

In section 2, we mentioned that the discrete weak space $S_h(j,l)$
and $\sum_h$ in the weak Galerkin method need to satisfy two
conditions. In \cite{JW_WG}, the authors proposed several possible
combinations of $S_h(j,l)$ and $\sum_h$. Through out this section we
use a uniform triangular mesh $\mathcal{T}_h$, the discrete weak
space $S_h(0,0)$, i.e., space consisting of piecewise constants on
the triangles and edges respectively, and $\sum_h$ with $V(T,1)$ to
be the lowest order Raviart-Thomas element $RT_0(T)$ in the weak
Galerkin method, which were used in \cite{NWG} for the numerical
studies of the weak Galerkin method for second-order elliptic
problems. We also adopt the various norms used in \cite{NWG} to
present the numerical results of the error $e_h$ between the $L^2$
projection $Q_h u$ of the exact solution and the numerical solution
$u_h$.

{\bf Example 1.} As the first example, we consider the following heat equation in $\Omega=(0,1)\times(0,1)$,
\begin{eqnarray}\label{example}
               u_t-\nabla\cdot(a\nabla u)  &=& f \quad \mbox{ in } \Omega,\quad \mbox{ for } t>0,\nonumber\\
               u &=& g\quad \mbox{ on }\partial\Omega, \quad \mbox{ for } t>0,\\
               u(\cdot, 0)&=&\psi\quad \mbox{ in } \Omega,\nonumber
             \end{eqnarray}
where $a=\left[ \begin{array} {cc} 1&0\\0&1\end{array}\right]$ and
$f$, $g$ and $\psi$ are determined by setting the exact solution
$u=\sin(2\pi(t^2+1)+\pi/2)\sin(2\pi x +\pi/2)\sin(2\pi y +\pi/2)$.

For this inhomogeneous Dirichlet boundary condition, with a uniform triangular
mesh $\mathcal{T}_h$, we chose the approximation space
$$
S_h=\left\{ v=\{v_0,v_b\}\,:\, \begin{array}r
v_0\in P_0(T),\quad\mbox{ for all } T\in \mathcal{T}_h,\\
v_b \in P_0(e) \mbox{ for all } T\in \mathcal{T}_h\mbox{ and } e\subset\partial T\notin \partial \Omega,\\
v_b=g_h \quad\mbox{ for all } T\in \mathcal{T}_h \mbox{ and } \partial T\in \partial \Omega
\end{array}
\right\}
$$
where $g_h$ is the $L^2$ projection of $g$ in the piecewise constant
finite element space on the boundary $\partial\Omega$. In the test,
$k=h$ and $k=h^2$ are used to check the order of convergency with
respect to time step size $k$ and mesh size $h$ respectively, since
the convergence rate of the error is dominated by that of the time
step size $k$ when $k=h$, and the convergence rate of the error is
dominated by that of the mesh size $h$ when $k=h^2$. The results are
shown in Table 1 and Table 2. Since the exact solution is smooth, we
observe the optimal convergence rates in both $L^2$ and weak
Galerkin $H^1$ norms for the Dirichlet boundary data type initial
boundary value problem, which is consistent with the theoretical
results shown in section 4 and 5.

\begin{table}[!!h]\label{tablesec}
\begin{center}
\caption{Convergence rate for heat equation with inhomogeneous Dirichlet boundary condition with $k=h$} 
\begin{tabular}{|c|c|c|c|c|c|}
\hline
$h$&${\|e_h\|}_{\{\infty,T\}}$&${\|e_h\|}_{\{\infty,\partial T\}}$&${\|\nabla_d e_h\|}$&${\|e_h\|}_{\{L^2,T\}}$ &${\|e_h\|}_{\{L^2,\partial T\}} $\\
\hline\hline
1/8&1.38e-01&1.44e-01&2.41e-01&4.90e-02&8.78e-02\\
1/16&6.97e-02&7.26e-02&8.34e-02&2.20e-02&4.03e-02\\
1/32&3.47e-02&3.56e-02&2.97e-02&1.05e-02&1.94e-02\\
1/64&1.72e-02&1.75e-02&1.10e-02&5.16e-03&9.54e-03\\
1/128&8.59e-03&8.65e-03&4.27e-03&2.56e-03   &4.74e-03\\
\hline\hline
$O(h^r)\,r=$&1.0012 &1.0138 & 1.4550&1.0643&1.0533\\
\hline
\end{tabular}
\end{center}
\end{table}

\begin{table}[!!h]\label{tablesec-2}
\begin{center}
\caption{Convergence rate for heat equation with inhomogeneous Dirichlet boundary condition with $k=h^2$} 
\begin{tabular}{||c||c||c||c||c||c||}
\hline\hline
$h$&${\|e_h\|}_{\{\infty,T\}}$&${\|e_h\|}_{\{\infty,\partial T\}}$&${\|\nabla_d e_h\|}$&${\|e_h\|}_{\{L^2,T\}}$ &${\|e_h\|}_{\{L^2,\partial T\}} $\\
\hline\hline
1/8&7.11e-02& 8.30e-02& 7.81e-02& 3.28e-02 &5.39e-02 \\
1/16&1.89e-02& 2.28e-02& 1.84e-02& 8.53e-03& 1.38e-02 \\
1/32&4.79e-03& 5.84e-03& 4.52e-03 &2.16e-03 &3.49e-03 \\
1/64&1.21e-03 &1.48e-03& 1.13e-03& 5.43e-04& 8.79e-04 \\
1/128&3.64e-04& 4.31e-04& 2.88e-04& 1.49e-04& 2.47e-04 \\
\hline\hline
$O(h^r)\,r=$&1.9025&1.8994 & 2.0213&1.9454&1.9418\\
\hline\hline
\end{tabular}
\end{center}
\end{table}

Next, we consider this heat problem with a mixed boundary condition
$$
\left\{ \begin{array}r
u=g\quad \mbox{ on }  \partial \Omega_1,\\
a\nabla u\cdot n+u=0, \quad\mbox{ on } \partial \Omega_2,
\end{array}
\right.
$$
where $\partial \Omega_1\cap\partial \Omega_2=\emptyset$, and
$\partial \Omega_1\cup\partial \Omega_2=\partial \Omega$. The exact
solution is set to be $u=\sin(2\pi(t^2+1)+\pi/2)\sin(\pi y)e^{-x}$,
where $\partial\Omega_2$ is the boundary segment $x=1$ and
$\partial\Omega_1$ is the union of all other boundary segments. For
the mixed boundary data type of initial boundary value problem, we
also achieved the optimal convergence rates of the error in all
norms as shown in Table 3 and Table 4.

\begin{table}[!!h]\label{tablesec-3}
\begin{center}
\caption{Convergence rate for heat equation with Robin boundary condition with $k=h$} 
\begin{tabular}{||c||c||c||c||c||c||}
\hline\hline
$h$&${\|e_h\|}_{\{\infty,T\}}$&${\|e_h\|}_{\{\infty,\partial T\}}$&${\|\nabla_d e_h\|}$&${\|e_h\|}_{\{L^2,T\}}$ &${\|e_h\|}_{\{L^2,\partial T\}} $\\
\hline\hline
1/8&1.65e-01&1.72e-01&1.80e-01&9.86e-02&1.83e-01\\
1/16&1.00e-01&1.02e-01&8.55e-02&5.86e-02&1.09e-01\\
1/32&5.54e-02&5.59e-02&4.01e-02&3.22e-02&5.95e-02\\
1/64&2.92e-02&2.93e-02&1.91e-02&1.69e-02&3.13e-02\\
1/128&1.50e-02&1.50e-02&9.24e-03&8.67e-03&1.60e-02\\
\hline\hline
$O(h^r)\,r=$&0.8656&0.8793&1.0713&0.8767&   0.8789\\
\hline\hline
\end{tabular}
\end{center}
\end{table}

\begin{table}[!!h]\label{tablesec-4}
\begin{center}
\caption{Convergence rate for heat equation with Robin boundary condition with $k=h^2$} 
\begin{tabular}{||c||c||c||c||c||c||}
\hline\hline
$h$&${\|e_h\|}_{\{\infty,T\}}$&${\|e_h\|}_{\{\infty,\partial T\}}$&${\|\nabla_d e_h\|}$&${\|e_h\|}_{\{L^2,T\}}$ &${\|e_h\|}_{\{L^2,\partial T\}} $\\
\hline\hline
1/8&3.18e-02&3.90e-02&2.61e-02&1.88e-02&3.57e-02\\
1/16&8.29e-03&1.01e-03&6.28e-03&4.87e-03&9.22e-03\\
1/32&2.10e-03&2.54e-03&1.55e-03&1.23e-03&2.32e-03\\
1/64&5.24e-04&6.34e-04&3.88e-04&3.08e-04&5.83e-04\\
1/128&1.39e-04&1.62e-04&1.06e-04&8.79e-05&1.65e-04\\
\hline\hline
$O(h^r)\,r=$&1.9591&1.9785&1.9875&1.9352&1.9383\\
\hline\hline
\end{tabular}
\end{center}
\end{table}

\medskip

{\bf Example 2.} For the second example, we consider the parabolic
problem (\ref{example}) of full tensor with Dirichlet boundary
condition and coefficient matrix $a=\left[ \begin{array} {cc}
x^2+y^2+1 &xy\\xy&x^2+y^2+1\end{array}\right]$, which is symmetric
and positive definite, and $f$, $g$ and $\psi$ are determined by
setting the exact solution $u=\sin(2\pi(t^2+1)+\pi/2)\sin(2\pi x
+\pi/2)\sin(2\pi y +\pi/2)$. The results are shown in Table 5 and
Table 6, which confirm the theoretical rates of convergence in
$L^2$. For the discrete $H^1$ norm, the numerical convergence is of
order $\mathcal{O}(h^2)$ which is one order higher than the
theoretical prediction. We believe that this suggests a
superconvergence between the weak Galerkin finite element
approximation and the $L^2$ projection of the exact solution.
Interested readers are encouraged to conduct a study on the
superconvergence phenomena.

\begin{table}[!!h]\label{tablesec-5}
\begin{center}
\caption{Convergence rate for parabolic problem with inhomogeneous Dirichlet boundary condition with $k=h$} 
\begin{tabular}{||c||c||c||c||c||c||}
\hline\hline
$h$&${\|e_h\|}_{\{\infty,T\}}$&${\|e_h\|}_{\{\infty,\partial T\}}$&${\|\nabla_d e_h\|}$&${\|e_h\|}_{\{L^2,T\}}$ &${\|e_h\|}_{\{L^2,\partial T\}} $\\
\hline\hline
1/8&1.21E-01&   1.38E-01&   3.57E-01&   4.64E-02&   8.14E-02\\
1/16&5.64E-02&  6.10E-02&   1.23E-01&1.87E-02   &3.39E-02\\
1/32&2.67E-02&  2.81E-02    &4.34E-02&  8.35E-03&   1.54E-02\\
1/64&1.29E-02   &1.33E-02&  1.55E-02&   3.95E-03&   7.29E-03\\
1/128&6.33E-03& 6.42E-03&   5.61E-03&   1.92E-03&   3.55E-03\\
\hline\hline
$O(h^r)\,r=$&1.0654&    1.1076  &1.4978&    1.1487& 1.1301\\
\hline\hline
\end{tabular}
\end{center}
\end{table}

\begin{table}[!!h]\label{tablesec-6}
\begin{center}
\caption{Convergence rate for parabolic problem with inhomogeneous Dirichlet boundary condition with $k=h^2$} 
\begin{tabular}{||c||c||c||c||c||c||}
\hline\hline
$h$&${\|e_h\|}_{\{\infty,T\}}$&${\|e_h\|}_{\{\infty,\partial T\}}$&${\|\nabla_d e_h\|}$&${\|e_h\|}_{\{L^2,T\}}$ &${\|e_h\|}_{\{L^2,\partial T\}} $\\
\hline\hline
1/8&7.41E-02&   9.68E-02&   1.24E-01&   3.53E-02&   5.74E-02\\
1/16&1.92E-02&  2.56E-02&   3.00E-02&   9.14E-03&   1.47E-02\\
1/32&4.83E-03&  6.44E-03&   7.45E-03&   2.30E-03&   3.70E-03\\
1/64&1.21E-03&  1.62E-03&   1.86E-03&   5.78E-04&   9.28E-04\\
1/128&3.35E-04& 4.34E-04&   4.67E-04&   1.53E-04&   2.48E-04\\
\hline\hline
$O(h^r)\,r=$&1.9474 &1.9500&    2.0118  &1.9637&    1.9636\\
\hline\hline
\end{tabular}
\end{center}
\end{table}

\end{document}